\documentclass[12pt]{article}
\usepackage[latin1]{inputenc}
\usepackage[T1]{fontenc}
\usepackage{amsmath,amssymb,amstext}
\usepackage{theorem}
\usepackage{a4wide}
\usepackage{multicol}
\usepackage[english]{babel}
\usepackage{enumerate}
\usepackage{eufrak}

\def\R{{ \mathbb{R}}}
\newtheorem{theorem}{Theorem}

\newtheorem{lemma}[theorem]{Lemma}

\newenvironment{proof}[1][Proof]{\textbf{#1.} }{\ \rule{0.5em}{0.5em}}
\def\cal{\mathcal}

\renewcommand{\geq}{\geqslant}
\def\leq{\leqslant}

\def\cal{\mathcal}
\def\1{{\mathbf{1}}}

\def\1{{\mathbf{1}}}
\def\0.5{{\frac{1}{2}}}

\renewcommand{\thefootnote}{\fnsymbol{footnote}}

\usepackage{mathrsfs}

\begin{document}
\renewcommand{\thefootnote}{\arabic{footnote}}
\begin{center}
{\Large{\bf Parameter estimation for SDEs related to stationary Gaussian processes }}
\\~\\
 Khalifa Es-Sebaiy\footnote{National
School of Applied Sciences - Marrakesh, Cadi Ayyad University,
Marrakesh, Morocco. Email: {\tt k.essebaiy@uca.ma}} and  Frederi G. Viens \footnote{%
Dept. Statistics and Dept. Mathematics, Purdue University, 150 N.
University
St., West Lafayette, IN 47907-2067, USA. E-mail: \texttt{viens@purdue.edu}}\\
{\it    Cadi Ayyad University  and Purdue University}\\~\\
\end{center}

{\small \noindent {\bf Abstract:} In this paper, we study   central and non-central
limit theorems for partial sum of functionals of general stationary Gaussian
fields. We apply our result to study drift parameter estimation problems
for some stochastic differential equations related to stationary Gaussian processes.\\
}

{\small \noindent {\bf Key words}: Central and non-central
limit theorems; stationary Gaussian processes; Stein's method;  drift parameter estimation; fractional Gaussian processes.
\\


\section{Introduction}

While the statistical inference of It\^o type diffusions has a long history, the statistical analysis for equations driven by fractional Brownian motion (fBm) is obviously more recent. The development of stochastic calculus with respect to the fBm allowed to study such models. We will recall  several approaches to estimate the parameters in fractional models but we mention that the below list is not exhaustive:
\begin{description}
\item{$\bullet$ } The MLE approach in \cite{KL}, \cite{TV}. In general the techniques used to construct
maximum likelihood estimators (MLE)  for the drift parameter
 are based on Girsanov transforms for fBm and depend on  the properties of the deterministic
fractional operators (determined by the Hurst parameter) related to the fBm. In general, the MLE is not easily computable.

\item{$\bullet$ }A least squares approach has been proposed in \cite{HN}. The study of the asymptotic properties of the estimator is based on certain criteria formulated in terms of the Malliavin calculus (see \cite{NP-book}). In the ergodic case, the statistical inference for several fractional Ornstein-Uhlenbeck (fOU) models has been recently developed  in the papers \cite{HN}, \cite{AM}, \cite{AV}, \cite{EEV}, \cite{HS}, \cite{CE}. The case of non-ergodic fOU process of the first kind and of the second kind  can be found in \cite{BEO} and  \cite{EET} respectively.
\end{description}

Our aim is to bring a new idea to develop  the statistical inference for stochastic differential equations related to stationary Gaussian processes  by proposing a suitable criteria. Our approach is based on Malliavin calculus and it makes in principle our estimator easier to be simulated. Moreover, the models studied in \cite{HN}, \cite{AM}, \cite{AV}, \cite{EEV} become particular cases in our approach (see section 4).

\section{Elements of Gaussian analysis and Malliavin calculus}
For the essential elements of Gaussian analysis and Malliavin calculus that are used in this paper see the  references (\cite{nualart-book}, \cite{NP-book}).

Now recall that, if $X, Y$ are two real-valued random variables, then the total variation distance between the law of $X$ and the law of $Y$ is given by
\[d_{TV}\left(X,Y\right)=\sup_{A\in \mathcal{B}(\R)}\left|P\left[X\in A\right]-P\left[Y\in A\right]\right|.\]
If $X, Y$ are two real-valued integrable random variables, then the Wasserstein distance between the law of $X$ and the law of $Y$ is given by
\[d_W\left(X,Y\right)=\sup_{f\in Lip(1)}\left|Ef(X)-Ef(Y)\right|\]
where $ Lip(1)$ indicates the collection of all Lipschitz functions with Lipschitz constant $\leq1$.\\

The following well-known direct consequence of the Borel-Cantelli Lemma (see
e.g. \cite{KN}), will allows us to turn convergence rates in the $p$-th mean
into pathwise convergence rates.
\begin{lemma}
\label{Borel-Cantelli} Let $\gamma >0$ and $p_{0}\in \mathbb{N}$.
Moreover let $(Z_{n})_{n\in \mathbb{N}}$ be a sequence of random variables.
If for every $p\geq p_{0}$ there exists a constant $c_{p}>0$ such that for
all $n\in \mathbb{N}$,
\begin{equation*}
(\mathbb{E}|Z_{n}|^{p})^{1/p}\leqslant c_{p}\cdot n^{-\gamma },
\end{equation*}%
then for all $\varepsilon >0$ there exists a random variable $\eta
_{\varepsilon }$ such that
\begin{equation*}
|Z_{n}|\leqslant \eta _{\varepsilon }\cdot n^{-\gamma +\varepsilon }\quad %
\mbox{almost surely}
\end{equation*}%
for all $n\in \mathbb{N}$. Moreover, $\mathbb{E}|\eta _{\varepsilon
}|^{p}<\infty $ for all $p\geq 1$.
\end{lemma}
\section{General context}
\subsection{Central and non-central limit central theorems}
 Consider  a centered stationary Gaussian process $Z=\left(Z_k\right)_{k\in \mathbb{Z}}$  with covariance
 \[r_Z(k):=E(Z_0Z_{k})\mbox{ and } |r_Z(k)|\leq1  \mbox{ for every } k\in\mathbb{Z}.\]
 Let us define the following normalized process
\begin{eqnarray}\label{expression of Y_i}
Y_k:=\frac{Z_k}{\sqrt{r_Z(0)}}, \quad k\in\mathbb{Z}.
\end{eqnarray}
We also define
\begin{eqnarray}\label{expression of V_{q,n}}
V_{q,n}(Z)&=&\sum_{k=1}^{q/2} I_{2k}(f_{2k,n})
\end{eqnarray}
where
\begin{eqnarray*}
f_{2k,n}&:=&d_{q,2k}(Z) \frac1{\sqrt{n}}\sum_{i=0}^{n-1}\varepsilon_i^{\otimes2k}
\end{eqnarray*}
with $Y_i=Y(\varepsilon_i)$ and $d_{q,2k}(Z)$ are constants.\\
If $\sum_{k\in \mathbb{Z}}\left|r_Y(k)\right|^2<\infty$,  we define
\[v_q(Z):= \lim_{n\rightarrow\infty}E\left[V_{q,n}^2(Z)\right].\]
We have the following central and non-central limit theorems.
\begin{theorem}\label{C-NC-LT}
Let  $\left(V_{2,n}(Z)\right)_{n\geq0}$  be the sequence defined  in
(\ref{expression of V_{q,n}}), we let $Y_k=\frac{Z_k}{\sqrt{r_Z(0)}},\ k\in\mathbb{Z}$ with covariance $r_Y(k)=E(Y_0Y_{k})$. Denote $N\sim\mathcal{N}(0,1)$.\\
1)  Then there exists a constant $C>0$ depending on $q$ and $r_Z(0)$ such that,
\begin{eqnarray}
 d_{TV}\left(\frac{V_{q,n}(Z)}{\sqrt{E\left[V_{q,n}^2(Z)\right]}},N\right)&\leq&  \frac{C}{E\left[V_{2,n}^2(Z)\right]}\sqrt{E\left[V_{2,n}^2(Y)\right] \sqrt{\kappa_{4}(V_{2,n}(Y))}+\kappa_{4}(V_{2,n}(Y)) }.\label{d_{TV}(F_{q,n}(Z),N)}
\end{eqnarray}
We also have in the case when, $\sum_{k\in \mathbb{Z}}\left|r_Y(k)\right|^2<\infty$,
\begin{eqnarray}
 &&d_{TV}\left(\frac{V_{q,n}(Z)}{\sqrt{v_q(Z)}},N\right)\nonumber\\&&\leq  \frac{C}{v_q(Z)}\left[\sqrt{E\left[V_{2,n}^2(Y)\right] \sqrt{\kappa_{4}(V_{2,n}(Y))}+\kappa_{4}(V_{2,n}(Y)) }+\left|v_q(Z)-E\left[V_{2,n}^2(Y)\right]\right|\right].\label{d_{TV}(F_{q,n}(Z),N)finite}
\end{eqnarray}
In addition:
\begin{itemize}
\item If $\sum_{k\in\mathbb{Z}}\left(r_Y(k)\right)^{2}<\infty$,
\begin{eqnarray}
 \kappa_{4}(V_{2,n}(Y))&\leq&  Cn^{-1}\left(\sum_{|k|<n}\left(r_Y(k)\right)^{4/3}\right)^{3}.\label{kappa_{4}(V_{2,n}(Y))}
\end{eqnarray}
\item If $r_Y(k)=k^{-\frac12}$,
\begin{eqnarray}
 \frac{\kappa_{4}(V_{2,n}(Y))}{\left(E\left[V_{2,n}^2(Y)\right]\right)^2}&\leq&  \frac{C}{\log(n)},\label{kappa_{4}(V_{2,n}(Y)) with log}
\end{eqnarray}
where $E\left[V_{2,n}^2(Z)\right]\sim 4d^2_{q,2}(Z)\log(n)$.
\end{itemize}
 2) Assume that $r_Y(k)=|k|^{-\alpha}$ with $0<\alpha<\frac12$. Then
\begin{eqnarray}\label{NCLT of Q_{q,n}(Z)}\frac{V_{q,n}(Z)}{n^{\frac12-\alpha}}\overset{law}{\longrightarrow}\frac{d_{q,2}(Z)}{\sqrt{D}}F_{\infty}
\end{eqnarray} where $D=\int_{-\infty}^{\infty}e^{iy}|y|^{\alpha-1}dy=2\Gamma(\alpha)\cos(\frac{\alpha\pi}{2})$, and
\[F_{\infty}=\int\int_{\R^2}e^{i(x+y)}\frac{e^{i(x+y)}-1}{i(x+y)}|xy|^{\frac{\alpha-1}{2}}W(dx)W(dy)\]
with $W$ is a random spectral measure of the white noise on $\R$.
\end{theorem}
\begin{proof}
Since $\frac{V_{q,n}(Z)}{\sqrt{E\left[V_{q,n}^2(Z)\right]}}\in\mathbb{D}^{1,2}$, by \cite[Proposition 2.4]{NP2013} we have
\begin{eqnarray*}
d_{TV}\left(\frac{V_{q,n}(Z)}{\sqrt{E\left[V_{q,n}^2(Z)\right]}},N\right)\leq 2 E\left|1-\left\langle D\frac{V_{q,n}(Z)}{\sqrt{E\left[V_{q,n}^2(Z)\right]}}, -DL^{-1}\frac{V_{q,n}(Z)}{\sqrt{E\left[V_{q,n}^2(Z)\right]}}\right\rangle_{\cal{H}}\right|.
\end{eqnarray*}
On the other hand, exploiting the fact that \[E\left[E\left[\left(I_{2k}(f_{2k,n})\right)^2\right]-\langle DI_{2k}(f_{2k,n}), -DL^{-1}I_{2k}(f_{2k,n})\rangle_{\cal{H}}\right]=0\] we obtain
\begin{eqnarray*}
&& E\left|1-\langle D\frac{V_{q,n}(Z)}{\sqrt{E\left[V_{q,n}^2(Z)\right]}}, -DL^{-1}\frac{V_{q,n}(Z)}{\sqrt{E\left[V_{q,n}^2(Z)\right]}}\rangle_{\cal{H}}\right|\\&& \leq \frac{1}{E\left[V_{q,n}^2(Z)\right]}\left(\sum_{k=1}^{q/2}\sqrt{Var\left((2k)^{-1}\|DI_{2k}(f_{2k,n})\|^2_{\cal{\mathcal{H}}}\right)} +\sum_{1\leq k<l\leq q/2}(2l)^{-1}E\left|\langle DI_{2k}(f_{2k,n}), DI_{2l}(f_{2l,n})\rangle_{\cal{H}}\right|\right.
 \\
 &&\left.+\sum_{1\leq l<k\leq q/2}(2l)^{-1}E\left|\langle DI_{2k}(f_{2k,n}), DI_{2l}(f_{2l,n})\rangle_{\cal{H}}\right|\right).
\end{eqnarray*}
Moreover, by \cite[Lemma 3.1]{NPP} we have
\begin{eqnarray*}
 \sqrt{Var\left((2k)^{-1}\|DI_{2k}(f_{2k,n})\|^2_{\cal{\mathcal{H}}}\right)}&=& (2k)^{-2}\sum_{j=1}^{2k-1}j^2j!^2\left(_j^{2k}\right)^4(4k-2j)!
 \|f_{2k,n}\underset{j}{\tilde{\otimes}}f_{2k,n}\|^2_{{\cal{\mathcal{H}}}^{\otimes 4k-2j}},
\end{eqnarray*}
and for $k<l$
\begin{eqnarray*}
&&E\left[\left((2l)^{-1}\langle DI_{2k}(f_{2k,n}), DI_{2l}(f_{2l,n})\rangle_{\cal{H}}\right)^2\right]\leq(2k)!\left(_{2k-1}^{2l-1}\right)^2(2l-2k)!
E\left[\left(I_{2k}(f_{2k,n})\right)^2\right] \|f_{2k,n}\underset{2l-2k}{{\otimes}}f_{2k,n}\|^2_{{\cal{\mathcal{H}}}^{\otimes 4k}}\\&&+
2k^{2}\sum_{j=1}^{2k-1}(l-1)!^2\left(_{j-1}^{2k-1}\right)^2\left(_{j-1}^{2l-1}\right)^2(2k+2l-2j)!
 \left(\|f_{2k,n}\underset{2k-j}{{\otimes}}f_{2k,n}\|^2_{{\cal{\mathcal{H}}}^{\otimes 2j}}+\|f_{2k,n}\underset{2l-j}{{\otimes}}f_{2k,n}\|^2_{{\cal{\mathcal{H}}}^{\otimes 2j}}\right).
 \end{eqnarray*}
Combining this together with the fact that for every $1\leq s \leq 2k-1$ with $k\in \{1,\ldots,q/2\}$
\begin{eqnarray*}
\|f_{2k,n}\underset{s}{\otimes}f_{2k,n}\|^2_{{\cal{\mathcal{H}}}^{\otimes 4k-2s}} &\leq& d_{q,2k}^4(Z) n^{-2}\sum_{k_1,k_2,k_3,k_4=1}^nr_Y^{s}(k_1-k_2)
r_Y^{s}(k_3-k_4)r_Y^{2k-s}(k_1-k_3)r_Y^{2k-s}(k_2-k_4)\\
&\leq& d_{q,2k}^4(Z) n^{-2}\sum_{k_1,k_2,k_3,k_4=1}^nr_Y(k_1-k_2)
r_Y(k_3-k_4)r_Y(k_1-k_3)r_Y(k_2-k_4)
\\&=& d_{q,2k}^4(Z) \kappa_{4}(V_{2,n}(Y))
\end{eqnarray*} we deduce that there exists a constant $C$ depending on $q$ and $r_Z(0)$ such that
\begin{eqnarray*}
 \sqrt{Var\left((2k)^{-1}\|DI_{2k}(f_{2k,n})\|^2_{\cal{\mathcal{H}}}\right)}&\leq& C \sqrt{\kappa_{4}(V_{2,n}(Y))}
\end{eqnarray*}
and
\begin{eqnarray*}
 (2l)^{-1}E\left|\langle DI_{2k}(f_{2k,n}), DI_{2l}(f_{2l,n})\rangle_{\cal{H}}\right|&\leq&
 \left(E\left[\left((2l)^{-1}\langle DI_{2k}(f_{2k,n}), DI_{2l}(f_{2l,n})\rangle_{\cal{H}}\right)^2\right]\right)^{1/2}
\\&\leq& C\sqrt{ E\left[\left(I_{2k}(f_{2k,n})\right)^2\right] \sqrt{\kappa_{4}(V_{2,n}(Y))}+\kappa_{4}(V_{2,n}(Y)) }
\\&\leq& C\sqrt{ n^{-1}\sum_{i,j=0}^{n-1}r_Y^{2k}(i-j) \sqrt{\kappa_{4}(V_{2,n}(Y))}+\kappa_{4}(V_{2,n}(Y)) }
\\&\leq& C\sqrt{ n^{-1}\sum_{i,j=0}^{n-1}r_Y^{2}(i-j) \sqrt{\kappa_{4}(V_{2,n}(Y))}+\kappa_{4}(V_{2,n}(Y)) }.
\end{eqnarray*}
Furthermore,
\begin{eqnarray*}
 d_{TV}\left(\frac{V_{q,n}(Z)}{\sqrt{E\left[V_{q,n}^2(Z)\right]}},N\right)
 &\leq& \frac{C}{E\left[V_{q,n}^2(Z)\right]}\sqrt{ n^{-1}\sum_{i,j=0}^{n-1}r_Y^{2}(i-j) \sqrt{\kappa_{4}(V_{2,n}(Y))}+\kappa_{4}(V_{2,n}(Y)) }
 \\
 &=& \frac{C}{E\left[V_{q,n}^2(Z)\right]}\sqrt{E\left[V_{2,n}^2(Y)\right] \sqrt{\kappa_{4}(V_{2,n}(Y))}+\kappa_{4}(V_{2,n}(Y)) }
 \\
 &\leq& \frac{C}{E\left[V_{2,n}^2(Z)\right]}\sqrt{ E\left[V_{2,n}^2(Y)\right] \sqrt{\kappa_{4}(V_{2,n}(Y))}+\kappa_{4}(V_{2,n}(Y)) }.
\end{eqnarray*}
Thus the estimate (\ref{d_{TV}(F_{q,n}(Z),N)}) is obtained. By a similar argument we obtain (\ref{d_{TV}(F_{q,n}(Z),N)finite}).\\
Moreover, if $\sum_{k\in\mathbb{Z}}\left(r_Y(k)\right)^{2}<\infty$, it follows from \cite{BBNP} that
\begin{eqnarray*}
 \kappa_{4}(V_{2,n}(Y))&\leq&  \frac{Cn^{-1}}{E\left[V_{2,n}^2(Y)\right]}\left(\sum_{|k|<n}\left(r_Y(k)\right)^{4/3}\right)^{3}
 \\&\leq&  Cn^{-1}\left(\sum_{|k|<n}\left(r_Y(k)\right)^{4/3}\right)^{3}.
\end{eqnarray*}
and if $r_Y(k)=k^{-\frac12}$ for $k\in\mathbb{Z}^*$, it follows from (\ref{expression of var(V_{q,n})}) that
\begin{eqnarray}
E\left[V_{q,n}^2(Z)\right]&=&d^2_{q,2}(Z) \frac{2}{n}\sum_{i,j=0}^{n-1}|r_Y(i-j)|^{2}+\sum_{k=2}^{q/2}d^2_{q,2k}(Z) \frac{(2k)!}{n}\sum_{i,j=0}^{n-1}|r_Y(i-j)|^{2k}\nonumber\\
&=&2d^2_{q,2}(Z)\left(1+ \frac{2}{n}\sum_{j=1}^{n-1}(n-j)j^{-1}\right)+\sum_{k=2}^{q/2}d^2_{q,2k}(Z) \frac{(2k)!}{n}\sum_{i,j=0}^{n-1}|r_Y(i-j)|^{2k}\nonumber\\
&=&2d^2_{q,2}(Z)\left(1+ 2\sum_{j=1}^{n-1}j^{-1}-\frac{2}{n}\sum_{j=1}^{n-1}j^{-2}\right)+\sum_{k=2}^{q/2}d^2_{q,2k}(Z) \frac{(2k)!}{n}\sum_{i,j=0}^{n-1}|r_Y(i-j)|^{2k}\nonumber\\
&\sim&4d^2_{q,2}(Z)\log(n).\label{log equivalence for var(V_{q,n})}
\end{eqnarray}
Thus, as in \cite{BN} we conclude that
\begin{eqnarray}
 \frac{\kappa_{4}(V_{2,n}(Y))}{\left(E\left[V_{2,n}^2(Y)\right]\right)^2}
 &=&\frac{\kappa_{4}(V_{2,n}(Y))}{\left(r_Z^{-1}(0)E\left[V_{2,n}^2(Z)\right]\right)^2}\nonumber\\
 &=&\frac{1}{n^2\left(r_Z^{-1}(0)E\left[V_{2,n}^2(Z)\right]\right)^2}
 \sum_{k_1,k_2,k_3,k_4=1}^nr_Y(k_1-k_2)
r_Y(k_3-k_4)r_Y(k_1-k_3)r_Y(k_2-k_4)\nonumber\\&\leq&\frac{C}{n^2\log^2(n)}  \sum_{k_1,k_2,k_3,k_4=1}^nr_Y(k_1-k_2)
r_Y(k_3-k_4)r_Y(k_1-k_3)r_Y(k_2-k_4)
 \nonumber\\&\leq&  \frac{C}{\log(n)}.\label{log equivalence for cumulant (V_{q,n})}
\end{eqnarray}
Hence the inequalities (\ref{kappa_{4}(V_{2,n}(Y))}) and (\ref{kappa_{4}(V_{2,n}(Y)) with log}) are satisfied.\\
The convergence (\ref{NCLT of Q_{q,n}(Z)}) is a direct consequence of \cite[Theorem 1]{DM}.
\end{proof}
\subsection{Hermite variation}
Let $q\in\mathbb{N}^*$ be even, and let $H_q$ be the pth Hermite polynomial. So, $H_q$ has the following decomposition \[H_q(x)=\sum_{k=0}^{\frac{q}{2}}\frac{q!(-1)^{k}}{k!(q-2k)!2^k} x^{q-2k}.\]
Define for every $q\geq2$   even,
 \begin{eqnarray}Q_{q,n}(Z):=\frac{1}{n}\sum_{k=0}^{n-1}H_q(Z_k)
\end{eqnarray}
and \[\gamma_Z(q):=E\left[H_q(Z_0)\right].\]
Then, we can write
\begin{eqnarray}
\gamma_Z(q)&:=&\sum_{k=0}^{\frac{q}{2}}\frac{q!(-1)^{k}}{k!(q-2k)!2^k}E\left(Z_0^{q-2k}\right)\nonumber
\\&=&\frac{q!}{ 2^{q/2}}\sum_{k=0}^{\frac{q}{2}}\frac{(-1)^{k}}{k!(\frac{q}{2}-k)!}\left[E\left(Z_0^2\right)\right]^{\frac{q}{2}-k}\nonumber
\\&=&\frac{q!}{(\frac{q}{2})! 2^{q/2}}\left(E\left(Z_0^2\right)-1\right)^{q/2}\label{limit a.s. of Q_{q,n}},
\end{eqnarray}
We have the following almost sure convergence.
\begin{theorem}Suppose that $Z$ is ergodic. Then, as $n\rightarrow\infty$
\begin{eqnarray}Q_{q,n}(Z)\longrightarrow\gamma_Z(q)\label{limit a.s. of Q_n{q,n}(Z)}
\end{eqnarray} almost surely.
\end{theorem}
Let $Y_i$ be the process defined in  (\ref{expression of Y_i}). We have,
\begin{eqnarray*}
H_q(Z_i)-EH_q(Z_i)=\sum_{k=1}^{q/2}b_{q,2k} H_{2k}(Y_i)
\end{eqnarray*}
where for any $k\in\{1,\ldots,\frac{q}{2}-1 \}$
\begin{eqnarray*}
b_{q,q-2k}(Z)&=&(-1)^{k}\left(r_Z^{\frac{q}{2}}(0)-r_Z^{\frac{q}{2}-1}(0)\right)a_{q-2}^{q}a_{q-4}^{q-2}\dots a_{q-2k}^{q-2k+2}\\
+&&(-1)^{k-1}\left(r_Z^\frac{q}{2}(0)-r_Z^{\frac{q}{2}-2}(0)\right)a_{q-4}^{q}a_{q-6}^{q-4}\ldots a_{q-2k}^{q-2k+2}\\
+&&\ldots
\\
+&&(-1)^{1}\left(r_Z^\frac{q}{2}(0)-r_Z^{\frac{q}{2}-k}(0)\right) a_{q-2k}^{q}
\end{eqnarray*}
and $b_{q,q}(Z)=r_Z^\frac{q}{2}(0)$, with for every $p$ even
\begin{eqnarray*}
a_{p-2k}^{p}=\frac{p!(-1)^{k}}{k!(p-2k)!2^k}\quad k=0,\ldots,p/2
\end{eqnarray*} which verify
\begin{eqnarray*}
H_p(x)=\sum_{k=0}^{p/2}a_{p-2k}^{p}x^{p-2k}.
\end{eqnarray*}
Define
\begin{eqnarray*}
V_{Q_{q,n}}(Z)&:=&\sqrt{n}\left(Q_{q,n}(Z)-\gamma_Z(q)\right).
\end{eqnarray*}
As consequence, we can write
\begin{eqnarray*}
V_{Q_{q,n}}(Z)&=&\sqrt{n}\left(Q_{q,n}(Z)-\gamma_Z(q)\right)\\&=&\frac1{\sqrt{n}}\sum_{i=0}^{n-1}
\left[ H_q(Z_i)-EH_q(Z_i)\right]\\&=& \frac1{\sqrt{n}}\sum_{i=0}^{n-1}
\sum_{k=1}^{q/2}b_{q,2k}(Z) H_{2k}(Y_i)
\\&=&\sum_{k=1}^{q/2}b_{q,2k}(Z) \frac1{\sqrt{n}}\sum_{i=0}^{n-1}
 I_{2k}\left(\varepsilon_i^{\otimes2k}\right)
 \\&=&\sum_{k=1}^{q/2} I_{2k}\left(b_{q,2k}(Z) \frac1{\sqrt{n}}\sum_{i=0}^{n-1}\varepsilon_i^{\otimes2k}\right)
\end{eqnarray*}
where $Y_i=Y(\varepsilon_i)$.\\
Furthermore,
\begin{eqnarray}\label{expression of var(V_{q,n})}
E\left[V_{Q_{q,n}}^2(Z)\right]&=&\sum_{k=1}^{q/2}b^2_{q,2k}(Z) \frac{(2k)!}{n}\sum_{i,j=0}^{n-1}|r_Y(i-j)|^{2k}\nonumber\\
&=&\sum_{k=1}^{q/2}b^2_{q,2k}(Z) (2k)!\left(1+\frac{2}{n}\sum_{j=1}^{n-1}(n-1-j)|r_Y(j)|^{2k}\right)\nonumber\\
&=&\sum_{k=1}^{q/2}b^2_{q,2k}(Z) (2k)!\left(1+2\sum_{j=1}^{n-1}|r_Y(j)|^{2k}-\frac{2}{n}\sum_{j=1}^{n-1}j|r_Y(j)|^{2k}\right).
\end{eqnarray}

We will need the following technical lemma.
\begin{lemma}\label{asymptotic R_{Q,q}} Let $(Z_k)_{k\geq0}$ be a stationary Gaussian sequence with $E(Z_0^2)<\infty$, and let $\lambda>0$ and  $q\in\mathbb{N}^*$  even.
Consider the sequence
\[R_{Q,q}(\lambda,Z_k):=H_q(Z_k-e^{-\lambda k}Z_0)-H_q(Z_k).\]
Then for every  $p\geq1$ there exits a constant $c(\lambda,q)$ depending on $\lambda, q$ and $E(Z_0^2)$ such that
\begin{eqnarray}\label{rate of sum of R_{Q,q}(Z_k) in Lp}\left\|\frac{1}{{n}}\sum_{k=0}^{n-1}R_{Q,q}(\lambda,Z_k)\right\|_{L^p(\Omega)}\leq \frac{c(\lambda,q)}{{n}}.
\end{eqnarray}
Moreover for every $\varepsilon>0$
\begin{eqnarray}\label{pathwise cv of sum of R_{Q,q}(Z_k)}\frac{1}{n^{\varepsilon}}\sum_{k=0}^{n-1}R_{Q,q}(\lambda,Z_k)\longrightarrow0
\end{eqnarray} almost surely as $n\rightarrow\infty$.
\end{lemma}
\begin{proof} We have
\begin{eqnarray*}R_{Q,q}(\lambda,Z_k)=\sum_{i=0}^{\frac{q}{2}}\frac{q!(-1)^{i}}{k!(q-2i)!2^i} \sum_{j=1}^{q-2i}(-1)^{j}(_j^{q-2i})e^{-\lambda jk}Z_0^jZ_k^{q-2i-j}.
\end{eqnarray*}Combining this with the fact that $Z$ is stationary and Gaussian, we obtain
\begin{eqnarray*}\left\|R_{Q,q}(\lambda,Z_k)\right\|_{L^p(\Omega)}\leq c_0(\lambda, q)e^{-\lambda k}
\end{eqnarray*}
Thus  (\ref{rate of sum of R_{Q,q}(Z_k) in Lp}) is obtained.\\
For the convergence (\ref{pathwise cv of sum of R_{Q,q}(Z_k)}), it is a direct consequence of  (\ref{rate of sum of R_{Q,q}(Z_k) in Lp}) and Lemma \ref{Borel-Cantelli}.
\end{proof}

Applying Theorem \ref{C-NC-LT} and Lemma \ref{asymptotic R_{Q,q}} we conclude the following result.
\begin{theorem}\label{CLT and NCLT of Q_{q,n}(Z-exp)}
Let  $\left(V_{Q_{q,n}}(Z)\right)_{n\geq0}$ and $\left(R_{Q,q}(\lambda,Z_k)\right)_{n\geq0}$ be the sequences defined  in
the above. \\
1) Then there exists $C$ depending on $\lambda$, $q$ and $r_Z(0)$ such that
\begin{eqnarray}
 &&d_{W}\left(\sqrt{\frac{n}{E\left[V_{Q_{q,n}}^2(Z)\right]}}\left(Q_{q,n}\left(Z_k-e^{-\lambda k}Z_0\right)-\gamma_{Z}(q)\right),N\right)\nonumber\\&\leq&  \frac{C}{E\left[V_{Q_{q,n}}^2(Z)\right]}\left(\sqrt{\frac{E\left[V_{Q_{q,n}}^2(Y)\right]}{n} }+\sqrt{E\left[V_{Q_{2,n}}^2(Y)\right] \sqrt{\kappa_{4}(V_{Q_{2,n}}(Y))}+\kappa_{4}(V_{Q_{2,n}}(Y)) }\right).\label{d_{W}(Q_{q,n}(Z-exp),N)}
\end{eqnarray}
On the other hand  if $\sum_{k\in \mathbb{Z}}\left|r_Y(k)\right|^2<\infty$, we can write
\begin{eqnarray}
 &&d_{W}\left(\sqrt{\frac{n}{v_{Q_q}(Z)}}\left(Q_{q,n}\left(Z_k-e^{-\lambda k}Z_0\right)-\gamma_{Z}(q)\right),N\right)\label{d_{W}(Q_{q,n}(Z-exp),N) finite}\\&\leq&  \frac{C}{v_{Q_q}(Z)}\left(\sqrt{\frac{v_{Q_q}(Z)}{n} }+\sqrt{E\left[V_{Q_{2,n}}^2(Y)\right] \sqrt{\kappa_{4}(V_{Q_{2,n}}(Y))}+\kappa_{4}(V_{Q_{2,n}}(Y)) } +\left|v_{Q_q}(Z)-E\left[V_{Q_{q,n}}^2(Y)\right]\right|\right)\nonumber
\end{eqnarray} where $v_{Q_q}(Z):= \lim_{n\rightarrow\infty}E\left[V_{Q_{q,n}}^2(Z)\right].$
Moreover,  $\kappa_{4}(V_{Q_{2,n}}(Y))$ verifies (\ref{kappa_{4}(V_{2,n}(Y))}) and (\ref{kappa_{4}(V_{2,n}(Y)) with log}).\\
 2) Assume that $r_Y(k)=|k|^{-\alpha}$ with $0<\alpha<\frac12$. Then
\begin{eqnarray}\label{NCLT of Q_{q,n}(Z-exp)}\frac{n^{\alpha}}{\sqrt{v_{Q_q}(Z)}}\left(Q_{q,n}\left(Z_k-e^{-\lambda k}Z_0\right)-\gamma_{Z}(q)\right)\overset{law}{\longrightarrow}\frac{b_{q,2}}{\sqrt{D}}F_{\infty}
\end{eqnarray} where $D$ and
$F_{\infty}$ are defined in Theorem \ref{C-NC-LT}.
\end{theorem}
\begin{proof}We have
\begin{eqnarray}
&&\sqrt{\frac{n}{E\left[V_{Q_{q,n}}^2(Z)\right]}}\left(Q_{q,n}\left(Z_k-e^{-\lambda k}Z_0\right)-\gamma_{Z}(q)\right)\nonumber\\&=&
\sqrt{\frac{n}{E\left[V_{Q_{q,n}}^2(Z)\right]}}\left(Q_{q,n}(Z)-\gamma_{Z}(q)\right)
+\frac{1}{\sqrt{nE\left[V_{Q_{q,n}}^2(Z)\right]}}\sum_{k=0}^{n-1}R_{Q,q}(\lambda,Z_k)
\nonumber\\&=&\frac{V_{Q_{q,n}}(Z)}{\sqrt{E\left[V_{Q_{q,n}}^2(Z)\right]}}
+\frac{1}{\sqrt{nE\left[V_{Q_{q,n}}^2(Z)\right]}}\sum_{k=0}^{n-1}R_{Q,q}(\lambda,Z_k).\label{decomposition of Q_{q,n}(Z-exp)}\end{eqnarray}
Hence,
\begin{eqnarray*}
&&d_{W}\left(\sqrt{\frac{n}{E\left[V_{Q_{q,n}}^2(Z)\right]}}\left(Q_{q,n}\left(Z_k-e^{-\lambda k}Z_0\right)-\gamma_{Z}(q)\right),N\right)\\&\leq&
 d_{W}\left(\frac{V_{Q_{q,n}}(Z)}{\sqrt{E\left[V_{Q_{q,n}}^2(Z)\right]}},N\right)
 +\left\|\frac{1}{\sqrt{nE\left[V_{Q_{q,n}}^2(Z)\right]}}\sum_{k=0}^{n-1}R_{Q,q}(\lambda,Z_k)\right\|_{L^1(\Omega)}.
\end{eqnarray*}
Combining this with (\ref{d_{TV}(F_{q,n}(Z),N)}) and (\ref{rate of sum of R_{Q,q}(Z_k) in Lp}) we obtain (\ref{d_{W}(Q_{q,n}(Z-exp),N)}). Similar argument leads to (\ref{d_{W}(Q_{q,n}(Z-exp),N) finite}).\\
Moreover, from (\ref{decomposition of Q_{q,n}(Z-exp)}), (\ref{NCLT of Q_{q,n}(Z)}) and (\ref{rate of sum of R_{Q,q}(Z_k) in Lp}) we deduce (\ref{NCLT of Q_{q,n}(Z-exp)}).
\end{proof}
\subsection{Power variation}
Consider, for every $q\in\mathbb{N}^*$  even, the following power variation
 \begin{eqnarray}P_{q,n}(Z):=\frac{1}{n}\sum_{i=0}^{n-1}(Z_i)^q.
\end{eqnarray}
Define
\begin{eqnarray}\delta_Z(q):=E\left[(Z_0)^q\right]=\frac{q!}{(\frac{q}{2})! 2^{q/2}}\left[E\left(Z_0^2\right)\right]^{q/2}\label{def. delta_Z(q)}.
\end{eqnarray}
We have the following almost sure convergence.
\begin{theorem}Suppose that $Z$ is ergodic. Then, as $n\rightarrow\infty$
\begin{eqnarray}P_{q,n}(Z)\longrightarrow\delta_Z(q)\label{limit a.s. of P{q,n}(Z)}
\end{eqnarray} almost surely.
\end{theorem}
Let $Y_i$ be the process defined in  (\ref{expression of Y_i}).
Let $c_{q,2k}=\frac{1}{(2k)!}\int_{-\infty}^{\infty}\frac{e^{-x^2/2}}{\sqrt{2\pi}}x^qH_{2k}(x)\ dx$ be the coefficients of the monomial $x^q$ expanded in the basis of
Hermite polynomials:
\begin{eqnarray*}
x^q=\sum_{k=0}^{q/2}c_{q,2k} H_{2k}(x).
\end{eqnarray*}
Then we can write,
\begin{eqnarray*}
V_{P_{q,n}}(Z)&=&\sqrt{n}\left(P_{q,n}(Z)-E\left[(Z_0)^q\right]\right)\\&=&\frac{[r_Z(0)]^{q/2}}{\sqrt{n}}\sum_{i=0}^{n-1}
\left(E\left[\left(\frac{Z_i}{\sqrt{r_Z(0)}}\right)^q\right]-E\left[\left(\frac{Z_0}{\sqrt{r_Z(0)}}\right)^q\right]\right)\\
&=& \frac{[r_Z(0)]^{q/2}}{\sqrt{n}}\sum_{i=0}^{n-1}
\sum_{k=1}^{q/2}c_{q,2k}(Z) H_{2k}(Y_i)
\\&=&\sum_{k=1}^{q/2}c_{q,2k}(Z) \frac{[r_Z(0)]^{q/2}}{\sqrt{n}}\sum_{i=0}^{n-1}
 I_{2k}\left(\varepsilon_i^{\otimes2k}\right)
 \\&=&\sum_{k=1}^{q/2} I_{2k}\left(c_{q,2k}(Z) \frac{[r_Z(0)]^{q/2}}{\sqrt{n}}\sum_{i=0}^{n-1}\varepsilon_i^{\otimes2k}\right)
\end{eqnarray*}
where $Y_i=Y(\varepsilon_i)=\frac{Z_i}{\sqrt{r_Z(0)}}$.
\\
Furthermore,
\begin{eqnarray}\label{expression of var(V_{q,n})}
E\left[V_{Q_{q,n}}^2(Z)\right]&=&[r_Z(0)]^{q}\sum_{k=1}^{q/2}c^2_{q,2k}(Z) \frac{(2k)!}{n}\sum_{i,j=0}^{n-1}|r_Y(i-j)|^{2k}\nonumber\\
&=&[r_Z(0)]^{q}\sum_{k=1}^{q/2}c^2_{q,2k}(Z) (2k)!\left(1+\frac{2}{n}\sum_{j=1}^{n-1}(n-1-j)|r_Y(j)|^{2k}\right)\nonumber\\
&=&[r_Z(0)]^{q}\sum_{k=1}^{q/2}c^2_{q,2k}(Z) (2k)!\left(1+2\sum_{j=1}^{n-1}|r_Y(j)|^{2k}-\frac{2}{n}\sum_{j=1}^{n-1}j|r_Y(j)|^{2k}\right).
\end{eqnarray}

We will also need the following technical lemma.
\begin{lemma}\label{asymptotic R_{P,q}} Let $(Z_k)_{k\geq0}$ be a stationary Gaussian sequence with $E(Z_0^2)<\infty$, and let $\lambda>0$ and  $q\in\mathbb{N}^*$  even.
Consider the sequence
\[R_{P,q}(\lambda,Z_k):=\left(Z_k-e^{-\lambda k}Z_0\right)^q-\left(Z_k\right)^q.\]
Then for every  $p\geq1$ there exits a constant $c(\lambda,q)$ depending on $\lambda, q$ and $E(Z_0^2)$ such that
\begin{eqnarray}\label{rate of sum of R_{P,q}(Z_k) in Lp}\left\|\frac{1}{{n}}\sum_{k=0}^{n-1}R_{P,q}(\lambda,Z_k)\right\|_{L^p(\Omega)}\leq \frac{c(\lambda,q)}{{n}}.
\end{eqnarray}
Moreover for every $\varepsilon>0$
\begin{eqnarray}\label{pathwise cv of sum of R_{P,q}(Z_k)}\frac{1}{n^{\varepsilon}}\sum_{k=0}^{n-1}R_{P,q}(\lambda,Z_k)\longrightarrow0
\end{eqnarray} almost surely as $n\rightarrow\infty$.
\end{lemma}
\begin{proof}The proof is straightforward by using similar arguments as in the proof of Lemma \ref{asymptotic R_{Q,q}}.
\end{proof}

Applying Theorem \ref{C-NC-LT} and Lemma \ref{asymptotic R_{P,q}} we conclude.
\begin{theorem}\label{CLT and NCLT of Q_{q,n}(Z-exp)}
Let  $\left(V_{P_{q,n}}(Z)\right)_{n\geq0}$ and $\left(R_{P,q}(\lambda,Z_k)\right)_{n\geq0}$ be the sequences defined  in
the above. \\
1) Then there exist $C$ depending on $\lambda$, $q$ and $r_Z(0)$ such that
\begin{eqnarray}
 &&d_{W}\left(\sqrt{\frac{n}{E\left[V_{P_{q,n}}^2(Z)\right]}}\left(P_{q,n}\left(Z_k-e^{-\lambda k}Z_0\right)-\delta_{Z}(q)\right),N\right)\nonumber\\&\leq&  \frac{C}{E\left[V_{P_{q,n}}^2(Z)\right]}\left(\sqrt{\frac{E\left[V_{P_{q,n}}^2(Y)\right]}{n} }+\sqrt{E\left[V_{P_{2,n}}^2(Y)\right] \sqrt{\kappa_{4}(V_{P_{2,n}}(Y))}+\kappa_{4}(V_{P_{2,n}}(Y)) }\right).\label{d_{W}(P_{q,n}(Z-exp),N)}
\end{eqnarray}
On the other hand  if $\sum_{k\in \mathbb{Z}}\left|r_Y(k)\right|^2<\infty$, we can write
\begin{eqnarray}
 &&d_{W}\left(\sqrt{\frac{n}{v_{P_q}(Z)}}\left(P_{q,n}\left(Z_k-e^{-\lambda k}Z_0\right)-\gamma_{Z}(q)\right),N\right)\label{d_{W}(P_{q,n}(Z-exp),N) finite}\\&\leq&  \frac{C}{v_{P_q}(Z)}\left(\sqrt{\frac{v_{P_q}(Z)}{n} }+\sqrt{E\left[V_{P_{2,n}}^2(Y)\right] \sqrt{\kappa_{4}(V_{P_{2,n}}(Y))}+\kappa_{4}(V_{P_{2,n}}(Y)) } +\left|v_q(Z)-E\left[V_{P_{q,n}}^2(Y)\right]\right|\right).\nonumber
\end{eqnarray}where $v_{P_q}(Z):= \lim_{n\rightarrow\infty}E\left[V_{P_{q,n}}^2(Z)\right].$
Moreover  $\kappa_{4}(V_{2,n}(Y))$ verifies (\ref{kappa_{4}(V_{2,n}(Y))}) and (\ref{kappa_{4}(V_{2,n}(Y)) with log}).\\
 2) Assume that $r_Y(k)=|k|^{-\alpha}$ with $0<\alpha<\frac12$. Then
\begin{eqnarray}\label{NCLT of P_{q,n}(Z-exp)}\frac{n^{\alpha}}{\sqrt{v_q(Z)}}\left(P_{q,n}\left(Z_k-e^{-\lambda k}Z_0\right)-\gamma_{Z}(q)\right)\overset{law}{\longrightarrow}\frac{c_{q,2}}{\sqrt{D}}F_{\infty}
\end{eqnarray} where $D$ and
$F_{\infty}$ are defined in Theorem \ref{C-NC-LT}.
\end{theorem}
\begin{proof}Similar proof as in Theorem \ref{CLT and NCLT of Q_{q,n}(Z-exp)}.
\end{proof}
\subsection{Quadradic case}
In this subsection we suppose that $q=2$. Then, in this case, we have
\begin{eqnarray*}V_{Q_{2,n}}(Z)=V_{P_{2,n}}(Z)&=&\frac{1}{\sqrt{n}}\sum_{k=0}^{n-1}\left(Z_k^2-E[Z_k^2]\right)\\
&=&\frac{E[Z_0^2]}{\sqrt{n}}\sum_{k=0}^{n-1}H_2(Y_k)=E[Z_0^2]V_{2,n}(Y)\\
&=&V_{2,n}(Z).
\end{eqnarray*}
Thus, we obtain the following theorem.
\begin{theorem}\label{CLT and NCLT of Q_{2,n}(Z-exp)}We have
\begin{itemize}
\item if $\sum_{k\in \mathbb{Z}}\left|r_Y(k)\right|^2<\infty$,    there exists $C>0$, $n_0\geq1$ such that for every $n\geq n_0$
\begin{eqnarray} d_{TV}\left(\frac{V_{2,n}(Z)}{\sqrt{E\left[V_{2,n}^2(Z)\right]}},N\right)\leq C\frac{\left(\sum_{|k|<n}|r_Y(k)|^{3/2}\right)^2}{\left(\sum_{|k|<n}|r_Y(k)|^2\right)^{3/2}\sqrt{n}},\label{d_{TV}(F_{2,n}(Z),N)}
\end{eqnarray}
and hence also \begin{eqnarray}
 d_{W}\left(\sqrt{\frac{n}{E\left[V_{2,n}^2(Z)\right]}}\left(Q_{2,n}\left(Z_k-e^{-\lambda k}Z_0\right)-\gamma_Z(2)\right),N\right)\leq  \frac{C}{\sqrt{n}}\left[1+\left(\sum_{|k|<n}|r_Y(k)|^{3/2}\right)^2\right];\label{d_{W}(F_{2,n}(Z-exp),N)}
\end{eqnarray}
\item if $|r_Y(k)|=|k|^{-\frac12}$,
\begin{eqnarray} d_{TV}\left(\frac{V_{2,n}(Z)}{\sqrt{E\left[V_{2,n}^2(Z)\right]}},N\right)\leq \frac{C}{(\log n)^{1/4}}; \label{d_{TV}(F_{2,n}(Z),N) case log}
\end{eqnarray}
and also
 \begin{eqnarray}
 d_{W}\left(\sqrt{\frac{n}{E\left[V_{2,n}^2(Z)\right]}}\left(Q_{2,n}\left(Z_k-e^{-\lambda k}Z_0\right)-\gamma_Z(2)\right),N\right)\leq  \frac{C}{(\log n)^{1/4}};\label{d_{W}(F_{2,n}(Z-exp),N) case log}
\end{eqnarray} where in this case $E\left[V_{2,n}^2(Z)\right]\sim 4b^2_{2,2}(Z)\log(n)=4r_{Z}^2(0)\log(n).$
\item if $r_Y(k)=|k|^{-\alpha}$ with $0<\alpha<\frac12$, there exists $ C>0$ depending on $\alpha$
\begin{eqnarray} d_{TV}\left(\frac{V_{2,n}(Z)}{\sqrt{E\left[V_{2,n}^2(Z)\right]}},\frac12\sqrt{\frac{1-2\alpha}{D}}F_{\infty}\right)\leq \frac{C}{\sqrt{\log n}},\label{d_{TV}(F_{2,n}(Z),F_{infty})}
\end{eqnarray}
and also
 \begin{eqnarray}
 d_{W}\left(\sqrt{\frac{n}{E\left[V_{2,n}^2(Z)\right]}}\left(Q_{2,n}\left(Z_k-e^{-\lambda k}Z_0\right)-\gamma_Z(2)\right),\frac12\sqrt{\frac{1-2\alpha}{D}}F_{\infty}\right)\leq \frac{C}{\sqrt{\log n}},\label{d_{W}(F_{2,n}(Z-exp),F_{infty})}
\end{eqnarray}
where $D$ and $F_{\infty}$ are defined in Theorem \ref{C-NC-LT}.\\
Moreover in this case $E\left[V_{2,n}^2(Z)\right]\sim 4\frac{n^{1-2\alpha}}{1-2\alpha}|r_Z(0)|^{2}.$
\end{itemize}
\end{theorem}
\begin{proof}Since $\frac{V_{2,n}(Z)}{\sqrt{E\left[V_{2,n}^2(Z)\right]}}=\frac{V_{2,n}(Y)}{\sqrt{E\left[V_{2,n}^2(Y)\right]}}$, then
(\ref{d_{TV}(F_{2,n}(Z),N)})  is a direct consequence of \cite[Theorem 3]{NV2014} (see also \cite{BBNP}, \cite{NP2013}). Hence also, from (\ref{decomposition of Q_{q,n}(Z-exp)}), (\ref{d_{TV}(F_{2,n}(Z),N)}) and (\ref{rate of sum of R_{Q,q}(Z_k) in Lp}) we deduce (\ref{d_{W}(F_{2,n}(Z-exp),N)}).\\
Combining (\ref{d_{TV}(F_{q,n}(Z),N)}) and (\ref{kappa_{4}(V_{2,n}(Y)) with log}) we obtain (\ref{d_{TV}(F_{2,n}(Z),N) case log}), and also the estimates (\ref{decomposition of Q_{q,n}(Z-exp)}), (\ref{d_{TV}(F_{2,n}(Z),N) case log}) and (\ref{rate of sum of R_{Q,q}(Z_k) in Lp}) lead to
 (\ref{d_{W}(F_{2,n}(Z-exp),N) case log}).\\
Now, suppose that $r_Y(k)=|k|^{-\alpha}$ with $0<\alpha<\frac12$. It is easy to see that (\ref{expression of var(V_{q,n})}) leads to
\begin{eqnarray*}E\left[V_{2,n}^2(Z)\right]&=& 2|r_Z(0)|^{2}\left(1+2\sum_{j=1}^{n-1}|r_Y(j)|^{2}-\frac{2}{n}\sum_{j=1}^{n-1}j|r_Y(j)|^{2k}\right)
\\&\sim&4\frac{n^{1-2\alpha}}{1-2\alpha}|r_Z(0)|^{2}.
\end{eqnarray*} Thus, from  \cite[Theorem 5]{NV2014} and  \cite[Theorem 1]{DM}, the estimate (\ref{d_{TV}(F_{2,n}(Z),F_{infty})}) is obtained. Hence also,
from (\ref{decomposition of Q_{q,n}(Z-exp)}), (\ref{d_{TV}(F_{2,n}(Z),F_{infty})}) and (\ref{rate of sum of R_{Q,q}(Z_k) in Lp}) we deduce
(\ref{d_{W}(F_{2,n}(Z-exp),F_{infty})})\end{proof}

\subsection{Improve the rate convergence}
Consider  a centered stationary Gaussian process $Z=\left(Z_k\right)_{k\in \mathbb{Z}}$  with covariance $r_Z(k)=E(Z_0Z_{k})<\infty$ and $|r_Z(k)|\leq1$  for $k\in\mathbb{Z}$.\\
Define the centered stationary Gaussian process $Z^{(p)}$ as follows: \[Z_k^{(1)}=Z_{k+1}-Z_{k},\quad k\in\mathbb{Z}\]
and for every $p\geq2$
 \[Z_k^{(p)}=Z_{k+1}^{(p-1)}-Z_k^{(p-1)}, k\in\mathbb{Z}.\]
Now, applying Theorem \ref{C-NC-LT} we conclude the following result.
\begin{theorem}
Assume that $k^{-\alpha}r_Z(k)$ converges to a constant for some $\alpha<-\frac12$.  Then for every $p\geq2$ there exists $C$ depending on $p, q$ and $r_Z(0)$ such that
\begin{eqnarray}
 d_{TV}\left(\frac{V_{Q_{q,n}}((Z^{(p)}))}{\sqrt{E\left[V_{Q_{q,n}}^2((Z^{(p)}))\right]}},N\right)&\leq&  Cn^{\alpha-2p+\frac52}.
\end{eqnarray}
In particular, if $\alpha=2H-2$ then  for every $p\geq2$
there exists $C$ depending on $p, q$ and $r_Z(0)$ such that
\begin{eqnarray}
 d_{TV}\left(\frac{V_{Q_{q,n}}((Z^{(p)}))}{\sqrt{E\left[V_{Q_{q,n}}^2((Z^{(p)}))\right]}},N\right)&\leq&  Cn^{2H-2p-\frac12}.
\end{eqnarray}
This leads that if $p\geq2$, then $\frac{V_{Q_{q,n}}((Z^{(p)}))}{\sqrt{E\left[V_{Q_{q,n}}^2((Z^{(p)}))\right]}}$ is asymptotically normal for every $H\in(0,1)$.
\end{theorem}
\section{Applications to Ornstein-Uhlenbeck processes}
\subsection{Fractional Ornstein-Uhlenbeck process }
In this section suppose that $%
X=\left\{X_t, t\geq0\right\}$ is an Ornstein-Uhlenbeck process
driven by a fractional Brownian motion $B^H=\left\{B^H_t, t\geq0\right\}$ of Hurst index $H\in(0,1)$. That is, $X$ is
the solution of the following linear stochastic
differential equation
\begin{equation}  \label{FOU}
 X_0=0; \quad dX_t=-\theta
X_tdt+dB^H_t,\quad t\geq0,
\end{equation} where
 whereas $\theta>0$ is considered as unknown parameter. \\
The solution $X$ has the following explicit expression:
\begin{eqnarray*}\label{expression of X_theta}
X_t=\int_{0}^{t}e^{-\theta(t-s)} dB^H_s.
\end{eqnarray*}
We can also write
\begin{eqnarray}\label{decomposition of X_theta}
X_t=Z^{\theta}_t-e^{-\theta t}Z^{\theta}_0
\end{eqnarray}
where
\begin{eqnarray*}\label{expression of Z_theta}
Z^{\theta}_t=\int_{-\infty}^{t}e^{-\theta(t-s)} dB^H_s.
\end{eqnarray*}
Moreover, $Y^{\theta}$ is an ergodic stationary Gaussian process.\\
We will need the following result.
 \begin{lemma}\label{inner product for OUFOU}Let $H\in(0,\frac12)\cup(\frac12,1]$, $m, m'>0$ and $-\infty\leq a<b\leq c<d<\infty$. Then
\begin{eqnarray*}E\left(\int_{a}^{b}e^{ms} dB^H(s)\int_{c}^{d}e^{m't} dB^H(t)\right)=H(2H-1)\int_{a}^{b}dse^{ms}  \int_{c}^{d}dte^{m't}(t-s)^{2H-2}
\end{eqnarray*}
\end{lemma}
\begin{proof}We use the same argument as in the proof of \cite[Lemma 2.1]{CKM}
\end{proof}
\\
\begin{lemma}\label{hypotheses FOU} Let $H\in(0,\frac12)\cup(\frac12,1)$, $m, m'>0$ and let $Z^{\theta}$ be the process defined
in (\ref{expression of Z_theta}). Then,
\begin{eqnarray}\label{first moment of Z_0}
E\left[\left(Z^{\theta}_0\right)^2\right]
=H\Gamma(2H)\theta^{-2H}
\end{eqnarray}
and for   large $|t|$
\begin{eqnarray}\label{inner product of Z_theta} E\left[Z^{\theta}_0Z^{\theta}_t\right]\sim\frac{H(2H-1)}{\theta^2}|t|^{2H-2}.
\end{eqnarray}
\end{lemma}
\begin{proof} see \cite[Theorem 2.3]{CKM} or Lemma \ref{hypotheses OUFOU}.
\end{proof}

\subsection{Construction  and asymptotic behavior of the estimators}
Fix $q\geq2$ and assume that $q$ is even.\\
From (\ref{decomposition of X_theta}) we can write
\begin{eqnarray}Q_{q,n}(X)&=&Q_{q,n}(Z^{\theta})+\frac{1}{n}\sum_{k=0}^{n-1}R_{Q,q}(\theta,Z^{\theta}_k)\label{decomposition of Q_{q,n}(X)}
\end{eqnarray}
Combining (\ref{decomposition of Q_{q,n}(X)}), Lemma \ref{asymptotic R_{Q,q}} and the fact that  $Z^{\theta}$ is ergodic we conclude that, almost surely,
\begin{eqnarray*}
\lim_{n\longrightarrow\infty}Q_{q,n}(X)\nonumber
&=&\lim_{n\longrightarrow\infty}Q_{q,n}(Z^{\theta}) \nonumber\\
&=&\gamma_{Z^{\theta}}(q)\nonumber\\
&=&\frac{q!}{(\frac{q}{2})! 2^{q/2}}\left(H\Gamma(2H)\theta^{-2H}-1\right)^{q/2}\nonumber
\\&:=&\mu_q(\theta).\label{limit a.s. 0f Q_{q,n}(X)}
\end{eqnarray*}
 Hence we obtain the following
estimator for $\theta$
\begin{eqnarray*}
\widehat{\theta}_{q,n}=\mu_q^{-1}\left[Q_{q,n}(X)\right].
\end{eqnarray*}
As consequence, we have the following strong consistence of $\widehat{\theta}_{q,n}$.
\begin{theorem}Let $H\in\left(0,1\right)$. Then, as $n\longrightarrow\infty$
\begin{eqnarray} \widehat{\theta}_{q,n} \longrightarrow \theta
\label{convergence a.s. of widehat{theta}_{q,n}}\end{eqnarray} almost surely.
\end{theorem}
Combining  (\ref{decomposition of Q_{q,n}(X)}) and Theorem  \ref{CLT and NCLT of Q_{q,n}(Z-exp)}  we conclude the following result.
\begin{theorem}\label{asymptotic distribution of Q_{q,n}(FOU)}
 Denote $N\sim\mathcal{N}(0,1)$. If $H\in(0,\frac34]$, then there exists $C$ depending on $q$, $H$ and $\theta$ such that
\begin{eqnarray}
 d_W\left(\sqrt{\frac{n}{E\left[V_{q,n}^2(Z^{\theta})\right]}}\left(\mu_q\left(\widehat{\theta}_{q,n}\right)-\mu_q(\theta)\right),N\right)&\leq&  C \left\lbrace\begin{aligned} n^{-\frac14},\quad
\mbox{if } 0<H<\frac58\\  n^{-\frac14}\log^{\frac34}(n),\quad
\mbox{if } H=\frac58
\\  n^{2H-\frac32},\quad
\mbox{if } \frac58<H<\frac34\\  \log^{-\frac14}(n),\quad
\mbox{if } H=\frac34. \end{aligned}%
\right.
\end{eqnarray}
In particular,
\begin{itemize}
\item If $H\in(0,\frac34)$
\begin{eqnarray}\sqrt{n}\left(\mu_q\left(\widehat{\theta}_{q,n}\right)-\mu_q(\theta)\right)\overset{law}{\longrightarrow}
\mathcal{N}\left(0,\sigma_q^2(Z^{\theta})\right)
\end{eqnarray}where $\sigma_q^2(Z^{\theta})=\sum_{k=1}^{q/2}b^2_{q,2k}(Z^{\theta}) (2k)!\left(1+2\sum_{j=1}^{\infty}\frac{|r_{Z^{\theta}}(j)|^{2k}}{r_{Z^{\theta}}^2(0)}\right)=\lim_{n\rightarrow\infty}E\left[V_{q,n}^2(Z^{\theta})\right]$,\\
\item if $H=\frac34$
\begin{eqnarray}\sqrt{\frac{n}{\log(n)}}\left(\mu_q\left(\widehat{\theta}_{q,n}\right)-\mu_q(\theta)\right)\overset{law}{\longrightarrow}
\mathcal{N}\left(0,4b_{q,2}^2(Z^{\theta})\right)
\end{eqnarray} where in this case $E\left[V_{q,n}^2(Z^{\theta})\right]\sim 4b_{q,2}^2(Z^{\theta})\log(n)$.
\end{itemize}
In the case when  $H\in(\frac34,1)$, we have
\begin{eqnarray}\frac{1}{n^{2H-\frac32}}\left(\mu_q\left(\widehat{\theta}_{q,n}\right)-\mu_q(\theta)\right)
\overset{law}{\longrightarrow}\frac{b_{q,2}(Z^{\theta})}{\sqrt{D}}F_{\infty}
\end{eqnarray} where $F_{\infty}$ is defined in Theorem \ref{C-NC-LT}.
\end{theorem}

Thus we deduce  the asymptotic distribution of $\widehat{\theta}_{q,n}$.
\begin{theorem}\label{asymptotic distrivution of widehat{theta}_{q,n}}
If $H\in(0, \frac{3}{4}]$, then
\begin{eqnarray}\sqrt{\frac{n}{E\left[V_{q,n}^2(Z^{\theta})\right]}}\left(\widehat{\theta}_{q,n}-\theta\right)\overset{law}{\longrightarrow}\mathcal{ N}\left(0, \left(\mu'_q(\theta)\right)^{-2}\right).
\end{eqnarray}
If  $H\in(\frac34,1)$, then
\begin{eqnarray}\frac{1}{n^{2H-\frac32}}\left( \widehat{\theta}_{q,n}  - \theta \right)
\overset{law}{\longrightarrow}\frac{b_{q,2}}{\mu'_q(\theta)\sqrt{D}}F_{\infty}
\end{eqnarray}
\end{theorem}
\begin{proof} We can write
\begin{eqnarray*}\sqrt{n}\left(\mu_q(\widehat{\theta}_{q,n})-\mu_q(\theta)\right)
&=&\mu'_q(\xi_{q,n})\sqrt{n}\left(\widehat{\theta}_{q,n}-\theta\right)
\end{eqnarray*}
where $\xi_{q,n}$ is a random variable between $\theta$ and $\widehat{\theta}_{q,n}$.\\
Combining this with Theorem \ref{asymptotic distribution of Q_{q,n}(FOU)} we obtain the desired conclusion.
\end{proof}

\vspace{.5cm}
{\bf Quadratic case}:
Now we will discuss the particular case when $q=2$.\\
Define
\begin{eqnarray}\label{expresion of (Y{theta}) depends on theta} Y^{\theta}_t:=\frac{Z_t^{\theta}}{\sqrt{\gamma_{Z^{\theta}}(2)}}=\frac{Z_t^{\theta}}{\sqrt{E[(Z^{\theta}_{0})^{2}]}}
=\frac{Z_t^{\theta}}{\sqrt{H\Gamma(2H)\theta^{-2H}}}.
\end{eqnarray}

Combining  (\ref{decomposition of Q_{q,n}(X)}) and Theorem  \ref{CLT and NCLT of Q_{2,n}(Z-exp)}  we conclude the following result.
\begin{theorem}If $H\in\left(0,\frac34\right]$, then
\begin{equation}  \label{d_W(F_{q,n}(X),N) for theta}
d_W\left(\frac{H\Gamma(2H)}{\sqrt{E\left[V_{2,n}^2(Z^{\theta})\right]}}\sqrt{n}\left(\widehat{\theta}_{2,n}^{-2H}-\theta^{-2H}\right),N\right)\leq C\times\left\lbrace\begin{aligned} \frac{1}{\sqrt{n}},\quad
\mbox{if } 0<H<\frac23\\  \frac{\log^2(n)}{\sqrt{n}},\quad
\mbox{if } H=\frac23\\  n^{6H-\frac92},\quad
\mbox{if } \frac23<H<\frac34\\  \frac{1}{(\log(n))^{1/4}},\quad
\mbox{if } H=\frac34. \end{aligned}%
\right.
\end{equation}
As consequence, for every $H\in\left(0,\frac34\right)$
\begin{eqnarray}\sqrt{n}\left(\widehat{\theta}_{2,n}-\theta\right)
\overset{\rm \mathcal{L}}{\longrightarrow}  \mathcal{ N}\left(0,\frac{\theta^{ 4H+2}\sigma_2^2(Z^{\theta})}{\left(H\Gamma(2H+1)\right)^2} \right) \label{convergence
in law theta}\end{eqnarray}
and if $H=\frac34$
\begin{eqnarray}\sqrt{\frac{n}{\log(n)}}\left(\widehat{\theta}_{2,n}-\theta\right)
\overset{\rm \mathcal{L}}{\longrightarrow}  \mathcal{ N}\left(0,\frac{H}{\theta}\right) \label{convergence
in law theta case log}\end{eqnarray}where in this case $E\left[V_{2,n}^2(Z^{\theta})\right]\sim \left(\Gamma(2H+1)\theta^{-2H}\right)^2\log(n)$.\\
If $\frac34<H<1$,
\begin{equation}\label{d_W(F_{q,n}(X),F_infty)}
d_W\left(\frac{H\Gamma(2H)}{\sqrt{E\left[V_{2,n}^2(Z^{\theta})\right]}}\sqrt{n}\left(\widehat{\theta}_{2,n}^{-2H}-\theta^{-2H}\right),
\frac12\sqrt{\frac{4H-3}{D}}F_{\infty}\right)\leq \frac{C}{\sqrt{\log(n)}}
\end{equation}where $E\left[V_{2,n}^2(Z^{\theta})\right] \sim \frac{n^{4H-3}}{4H-3}\left(\Gamma(2H+1)\theta^{-2H}\right)^{2}.$\\
As consequence
\begin{eqnarray}\sqrt{n}\left(\theta-\widehat{\theta}_{2,n}\right)
\overset{\rm {law}}{\longrightarrow}  \frac{\theta}{2H}F_{\infty}. \label{convergence
in law thetaRosenblatt}\end{eqnarray}
\end{theorem}

\subsection{OU driven by fractional Ornstein-Uhlenbeck process }
In this section suppose that $%
X=\left\{X_t, t\geq0\right\}$ is an Ornstein-Uhlenbeck process
driven by fractional Ornstein-Uhlenbeck process $V=\left\{V_t,
t\geq0\right\}$ given by the following linear stochastic
differential equations
\begin{equation}  \label{OUFOU}
\left\lbrace\begin{aligned} X_0=0; \quad dX_t=-\theta
X_tdt+dV_t,\quad
t\geq0\\ V_0=0;\quad dV_t=-\rho V_tdt+dB^H_t,\quad t\geq0, \end{aligned}%
\right.
\end{equation} where
$B^H=\left\{B^H_t, t\geq0\right\}$ is a fractional Brownian motion
of Hurst index $H\in(0,1)$, whereas $\theta>0$ and
$\rho>0$ are considered as unknown parameters such that
$\theta\neq\rho$. \\
For \cite{EEV} we have the following results related to $X_t$:
\begin{equation}\label{representationX}
X_t=\frac{\rho}{\rho-\theta}X^{\rho}_t+\frac{\theta}{\theta-\rho}X^\theta_t
\end{equation}
 where for
$m>0$
\begin{eqnarray}\label{expression of X_m}
X^{m}_t=\int_{0}^{t}e^{-m(t-s)} dB^H_s.
\end{eqnarray} On the other hand, we can also write the system
(\ref{OUFOU}) as follows
\begin{equation*}
dX_t =-\left( \theta +\rho \right) X_t dt-\rho \theta \Sigma_t
dt+dB^H_t . \label{SIDE}
\end{equation*}where for $0\leq t \leq T$
\begin{eqnarray}
\Sigma_t= \int_{0}^{t}X_s ds =\frac{V_t-X_t}{\theta}
=\frac{X^{\theta}_t-X^{\rho}_t}{\rho-\theta}  \label{expression
Sigma}
\end{eqnarray}
We can also write
\begin{eqnarray}\label{decomposition of X_m}
X^{m}_t=Z^m_t-e^{-mt}Z^m_0
\end{eqnarray}
where
\begin{eqnarray}\label{expression of Z_m}
Z^m_t=\int_{-\infty}^{t}e^{-m(t-s)} dB^H_s.
\end{eqnarray}
As consequence
\begin{eqnarray}\label{representationX with Z}
X_t&=&\frac{\rho}{\rho-\theta}Z^{\rho}_t+\frac{\theta}{\theta-\rho}Z^\theta_t-\left(\frac{\rho e^{-\rho t}}{\rho-\theta}Z^{\rho}_0+\frac{\theta e^{-\theta t}}{\theta-\rho}Z^\theta_0\right)\nonumber\\&:=&Z^{\theta,\rho}_t+\frac{\rho}{\rho-\theta}R(\rho,Z^{\rho})+\frac{\theta }{\theta-\rho}R(\theta,Z^\theta)
\end{eqnarray}
and
\begin{eqnarray}\label{expression
Sigma with Z}
\Sigma_t&=&\frac{Z^{\theta}_t-Z^{\rho}_t}{\rho-\theta}- \frac{e^{-\theta t}Z^{\theta}_0-e^{-\rho t}Z^{\rho}_0}{\rho-\theta}\nonumber  \\&:=&
\Sigma^{\theta,\rho}_t+\frac{1}{\rho-\theta}\left(R(\theta,Z^\theta)-R(\rho,Z^\rho)\right).
\end{eqnarray}
On the other hand, the process
\begin{eqnarray}\label{ergodic stationary Y(theta,rho)}
\left(Z^m_t,Z^{m'}_t\right)
\end{eqnarray} is an ergodic stationary Gaussian process.
\begin{lemma}\label{hypotheses OUFOU} Let $H\in(0,\frac12)\cup(\frac12,1)$, $m, m'>0$ and let $Z^m$ be the process defined
in (\ref{expression of Z_m}). Then,
\begin{eqnarray}\label{first moment of Z_0}
\lambda(m,m'):=E\left[Z^{m}_0Z^{m'}_0\right]
=\frac{H\Gamma(2H)}{m+m'}\left(m^{1-2H}+{(m')}^{1-2H}\right)
\end{eqnarray}
and for   large $|t|$
\begin{eqnarray}\label{inner product of Z} E\left[Z^{m}_0Z^{m'}_t\right]\sim\frac{H(2H-1)}{mm'}|t|^{2H-2}.
\end{eqnarray}
\end{lemma}
This implies that for $H\in(0,\frac12)\cup(\frac12,1)$
\begin{equation*}
\eta_X(\theta,\rho):=E\left[\left(Z_0^{\theta,\rho}\right)^2\right]=\frac{H\Gamma(2H)}{\rho^2-\theta^2}[\rho^{2-2H}-%
\theta^{2-2H}],
\end{equation*}
\begin{equation*}
\eta_{\Sigma}(\theta,\rho):=E\left[\left(\Sigma_0^{\theta,\rho}\right)^2\right]=\frac{H\Gamma(2H)}{\rho^2-\theta^2}%
[\theta^{-2H}-\rho^{-2H}],
\end{equation*}
and
\begin{equation*}
 E\left[\left(Z_0^{\theta,\rho}\Sigma_0^{\theta,\rho}\right) \right]=0.
\end{equation*}
\begin{proof}By using \cite[Proposition A.1]{CKM}, we can write
\begin{eqnarray*}
E\left[Z^{m}_0Z^{m'}_0\right]&=&mm'\int_{-\infty}^{0}\int_{-\infty}^{0}e^{mu}  e^{m'v} E\left(B^H_uB^H_v\right)\ dudv\\
&=&\frac{mm'}{2}\int_{-\infty}^{0}\int_{-\infty}^{0}e^{mu}  e^{m'v}  \left( (-u)^{2H} + (-v)^{2H}-|v-u|^{2H}\right)\ dudv
\\
&=&\frac{mm'}{2}\int_{0}^{\infty}\int_{0}^{\infty}e^{mu}  e^{m'v}  \left( u^{2H} +v^{2H}-|v-u|^{2H}\right)\ dudv
\\
&=&\frac{\Gamma(2H+1)}{2(m+m')}\left(m^{1-2H}+{(m')}^{1-2H}\right).
\end{eqnarray*}
Thus the estimate (\ref{first moment of Z_0}) is proved.
\\Now, let $0<\varepsilon<1$
\begin{eqnarray*}
E\left(Z^m_0Z^{m'}_t\right)&=&e^{-m't}E\left(\int_{-\infty}^{0}e^{mu} dB^H_u\int_{-\infty}^{t}e^{m'v} dB^H_v\right)\\&=&
e^{-m't}E\left(\int_{-\infty}^{0}e^{mu} dB^H_u\int_{-\infty}^{\varepsilon t}e^{m'v} dB^H_v\right)+e^{-m't}E\left(\int_{-\infty}^{0}e^{mu} dB^H_u\int_{\varepsilon t}^{t}e^{m'v} dB^H_v\right)
\\&:=&A+B
\end{eqnarray*}
where, using \cite[Proposition A.1]{CKM} it is easy to see that
\begin{eqnarray*}
|A|=O\left(e^{-m't}\right).
\end{eqnarray*}
On the other hand, by Lemma \ref{inner product for OUFOU} and integrations by part
\begin{eqnarray*}
 B&=& H(2H-1)e^{-m't} \int_{-\infty}^{0}du\ e^{mu} \int_{\varepsilon t}^{t}dv\ e^{m'v} (v-u)^{2H-2}
\\&=& H(2H-1)e^{-m't} \int_{-\infty}^{0}du\ e^{mu} \int_{\varepsilon t-u}^{t-u}dz\ e^{m'(u+z)}z^{2H-2}
\\&=& H(2H-1)e^{-m't} \int_{\varepsilon t}^{\infty}dz\ e^{m'z}z^{2H-2}\int_{\varepsilon t-z}^{0\wedge(t-z)}du\ e^{(m+m')u}
\\&=&\frac{H(2H-1)}{m+m'}\left(\int_{t}^{\infty} e^{-m(z-t)}z^{2H-2} dz+  \int_{\varepsilon t}^{t} e^{-m'(t-z)}z^{2H-2}dz+e^{-m't(1-\varepsilon)} \int_{\varepsilon t}^{\infty} e^{-m(z-\varepsilon t)}z^{2H-2}dz\right)
\\&=&\frac{H(2H-1)}{(m+m')}\left(\frac{t^{2H-2}}{m}+\frac{2H-2}{m}\int_{t}^{\infty} e^{-m(z-t)}z^{2H-3} dz+ \frac{t^{2H-2}}{m'}\right.\\&&\left. - \frac{(\varepsilon t)^{2H-2}}{m'}e^{-m'(1-\varepsilon)t}-\frac{2H-2}{m'}\int_{\varepsilon t}^{t} e^{-m'(t-z)}z^{2H-3}dz+ e^{-m't(1-\varepsilon)} \int_{\varepsilon t}^{\infty} e^{-m(z-\varepsilon t)}z^{2H-2}dz\right)
\\&=&\frac{H(2H-1)}{mm'}t^{2H-2}+ o\left(t^{2H-2}\right),
\end{eqnarray*}
the last inequality comes from the fact that
\begin{eqnarray*}
\int_{t}^{\infty} e^{-m(z-t)}z^{2H-3} dz&\leq& t^{-1}\int_0^{\infty} e^{-my}dy
\\&&\rightarrow 0, \quad \mbox{ as } t\rightarrow\infty,
\end{eqnarray*}
\begin{eqnarray*}
 t^{2-2H}\int_{\varepsilon t}^{t} e^{-m'(t-z)}z^{2H-3}dz&\leq& \varepsilon^{2H-3}t^{-1}\int_{\varepsilon t}^{t} e^{-m'(t-z)}dz\\
 &=& \varepsilon^{2H-3}t^{-1}\int_0^{(1-\varepsilon) t} e^{-m'y}dy
 \\&&\rightarrow 0, \quad \mbox{ as } t\rightarrow\infty,
\end{eqnarray*}
and
\begin{eqnarray*}
 t^{2-2H}e^{-m't(1-\varepsilon)}\rightarrow 0, \quad \mbox{ as } t\rightarrow\infty.
\end{eqnarray*}
So, we conclude that the estimate (\ref{inner product of Z}) is obtained.
\end{proof}
\subsection{Construction  and asymptotic behavior of the estimators}

{\bf\large \noindent Case:  $X$ and $\Sigma$ are observed}\\

Combining (\ref{representationX with Z}), (\ref{expression Sigma with Z}), (\ref{limit a.s. of Q_{q,n}}), (\ref{ergodic stationary Y(theta,rho)}) and Lemma \ref{asymptotic R_{Q,q}} we conclude that
\begin{eqnarray*}
\lim_{n\longrightarrow\infty}\left(Q_{q,n}(X),Q_{q,n}(\Sigma)\right)\nonumber
&=&\lim_{n\longrightarrow\infty}\left(Q_{q,n}(Y^{\theta,\rho}),Q_{q,n}(\Sigma^{\theta,\rho})\right) \nonumber\\
&=&\left(\gamma_{Y^{\theta,\rho}}(q),\gamma_{\Sigma^{\theta,\rho}}(q)\right)\nonumber\\
&=&\frac{q!}{(\frac{q}{2})! 2^{q/2}}\left(\left(\eta_{X}(\theta,\rho)-1\right)^{q/2},\left(\eta_{\Sigma}(\theta,\rho)-1\right)^{q/2}\right)\nonumber
\\&:=&F_q(\theta,\rho).\label{limit a.s. 0f Q_{q,n}(X,Sigma)}
\end{eqnarray*}
 Hence, we obtain the following
estimator for $(\theta,\rho)$.
\begin{eqnarray*}
\left(\widehat{\theta}_{q,n},\widehat{\rho}_{q,n}\right)=G_q\left(Q_{q,n}(X),Q_{q,n}(\Sigma)\right).
\end{eqnarray*}where $G_q$ is the inverse function of  $F_q$.\\
By construction, we have the following strong consistence of $\left(\widehat{\theta}_{q,n},\widehat{\rho}_{q,n}\right)$.
\begin{theorem}Let $H\in\left(0,1\right)$. Then, as $n\longrightarrow\infty$
\begin{eqnarray} \left(\widehat{\theta}_{q,n},\widehat{\rho}_{q,n}\right) \longrightarrow (\theta,\rho)
\label{convergence a.s. of widehat{theta,rho}_{q,n}}\end{eqnarray} almost surely.
\end{theorem}
Combining  (\ref{expression
Sigma with Z}), (\ref{representationX with Z}), Lemma \ref{hypotheses OUFOU} and Theorem \ref{CLT and NCLT of Q_{q,n}(Z-exp)} we conclude the following result
\begin{theorem}Let $H\in\left(0, \frac34\right)$. Then
\begin{eqnarray}\sqrt{n}\left[\left(Q_{q,n}(X),Q_{q,n}(\Sigma)\right)-F_q(\theta,\rho)\right]
\overset{\rm \mathcal{L}}{\longrightarrow}  \mathcal{ N}\left(0,  \Gamma(\theta,\rho) \right)
\label{convergence in law vector Q}\end{eqnarray} where
\begin{eqnarray}\label{matrix Gamma(theta,rho)}
\Gamma(\theta,\rho)=\left(%
\begin{matrix}
\sigma_q^2(Y^{\theta,\rho}) & \sigma_q\left(Y^{\theta,\rho},\Sigma^{\theta,\rho}\right)  \\
\sigma_q\left(Y^{\theta,\rho},\Sigma{\theta,\rho}\right)  & \sigma_q^2(\Sigma^{\theta,\rho})
\end{matrix}%
\right)
\end{eqnarray}%
where
\begin{eqnarray*}
&&\sigma_q^2(Y^{\theta,\rho})= Var\left[H_q(Y^{\theta,\rho}_0)\right]+2\sum_{k=1}^{\infty}Cov\left[H_q(Y^{\theta,\rho}_0),H_q(Y^{\theta,\rho}_{k})\right]\\
&&\sigma_q^2(\Sigma^{\theta,\rho})= Var\left[H_q(\Sigma^{\theta,\rho}_0)\right]+2\sum_{k=1}^{\infty}Cov\left[H_q(\Sigma^{\theta,\rho}_0),H_q(\Sigma^{\theta,\rho}_{k})\right]\\
&&\sigma_q\left(Y^{\theta,\rho},\Sigma^{\theta,\rho}\right)=Cov\left[H_q(Y^{\theta,\rho}_0),H_q(\Sigma^{\theta,\rho}_{0})\right]
+\sum_{k\in \mathbb{Z}*}Cov\left[H_q(Y^{\theta,\rho}_0),H_q(\Sigma^{\theta,\rho}_{k})\right].
\end{eqnarray*}
\end{theorem}

\begin{theorem}Let $H\in\left(0,\frac34\right)$. Then
\begin{eqnarray}\sqrt{n}\left(\widehat{\theta}_{q,n}-\theta,\widehat{\rho}_{q,n}-\rho\right)
\overset{\rm \mathcal{L}}{\longrightarrow}  \mathcal{ N}\left(0,\
J_{G_q}(\eta_X,\eta_{\Sigma})  \ \Gamma(\theta,\rho)\
  ^tJ_{G_q}(\eta_X(\theta,\rho),\eta_{\Sigma}(\theta,\rho))\right) \label{convergence
in law vector (theta,rho)}\end{eqnarray}
\end{theorem}
\begin{proof}
By Taylor's
formula we can write
\begin{eqnarray*}\sqrt{n}\left(\widehat{\theta}_{q,n}-\theta,\widehat{\rho}_{q,n}-\rho\right)
&=&\sqrt{n}\left(Q_{q,n}(X)-\gamma_{Y^{\theta,\rho}}(q),Q_{q,n}(\Sigma)-\gamma_{\Sigma^{\theta,\rho}}(q)\right)\
^tJ_{G_q}(\eta_X,\eta_{\Sigma})+ d_n
\end{eqnarray*}
where $d_n$  converges in distribution to zero,
because
\begin{eqnarray*}
\|d_n\|\leqslant c(\theta,\rho,H)\sqrt{T_n}\|Q_{q,n}(X)-\gamma_{Y^{\theta,\rho}}(q),%
Q_{q,n}(\Sigma)-\gamma_{\Sigma^{\theta,\rho}}(q))\|^2.
\end{eqnarray*}
It is easy to see that if for any $w\in\Omega$ there exists $n_0(w)\in%
\mathbb{N}$ such that $X_n(w)=Y_n(w)$ for all $n\geq n_0(w)$ and $X_n\overset%
{law}{\longrightarrow}0$ as $n\rightarrow\infty$, then $Y_n\overset{law}{%
\longrightarrow}0$ as $n\rightarrow\infty$.\newline
Combining this and  above estimates we obtain the desired result.\end{proof}

\vspace{.5cm}
{\bf Quadratic case}:
Here we suppose that $q=2$. In this case we have:\\
 $F_2$ is a positive function of the
variables $\left( x,y\right) $ in $(0,+\infty )^{2}$ defined by: for every $(x,y)\in(0,+\infty )^{2}$
\begin{equation}
F_2(x,y)=H\Gamma(2H)\times\left\{
\begin{array}{lcl}
\frac{1}{y^2-x^2}
\left(y^{2-2H}-x^{2-2H},x^{-2H}-y^{-2H}\right)\quad\mbox{if }\ x\neq y  \\
\left((1-H)x^{-2H},Hx^{-2H-2}\right)\quad\mbox{if }\ y=x.
\end{array}%
\right.   \label{new system with Sigma}
\end{equation}
Since for every $(x,y)\in(0,+\infty )^{2}$ with $x\neq y$ \begin{equation*}
J_{F_2}\left( x,y\right) =%
\Gamma(2H+1)\begin{pmatrix}
\frac{\left( 1-H\right) x^{1-2H}\left( x^2-y^2\right)
-x\left(x^{2-2H}-y^{2-2H}\right)}{\left( x^2-y^2\right) ^{2}} & \frac{\left( 1-H\right) y^{1-2H}\left( y^2-x^2\right)
-y\left(y^{2-2H}-x^{2-2H}\right)}{\left( x^2-y^2\right) ^{2}} \\
\frac{Hx^{-2H-1}\left( x^2-y^2\right)+x\left( x^{-2H}-y^{-2H}\right) }{\left( x^2-y^2\right)
^{2}} &
\frac{Hy^{-2H-1}\left( y^2-x^2\right)+y\left( y^{-2H}-x^{-2H}\right) }{\left( x^2-y^2\right)
^{2}}
\end{pmatrix}%
  \label{Jacobian}
\end{equation*}%
the determinant of $J_{F_2}\left( x,y\right) $ is non-zero on in
$(0,+\infty )^{2} $. So, $F_2$ is a diffeomorphism in $(0,+\infty
)^{2} $ and its inverse $G_2$ has a Jacobian%
\begin{equation*}
J_{G_2}\left( a,b\right) =\frac{\Gamma(2H+1)}{\det J_{F_2}\left( x,y\right) }%
\begin{pmatrix}
\frac{Hy^{-2H-1}\left( y^2-x^2\right)+y\left( y^{-2H}-x^{-2H}\right) }{\left( x^2-y^2\right)
^{2}} & -\frac{\left( 1-H\right) y^{1-2H}\left( y^2-x^2\right)
-y\left(y^{2-2H}-x^{2-2H}\right)}{\left( x^2-y^2\right) ^{2}} \\
-\frac{Hx^{-2H-1}\left( x^2-y^2\right)+x\left( x^{-2H}-y^{-2H}\right) }{\left( x^2-y^2\right)
^{2}} &\frac{\left( 1-H\right) x^{1-2H}\left( x^2-y^2\right)
-x\left(x^{2-2H}-y^{2-2H}\right)}{\left( x^2-y^2\right) ^{2}}
\end{pmatrix}%
;
\end{equation*}where $\left( x,y\right) =G_2\left( a,b\right)$.

\subsection{  Fractional Ornstein-Uhlenbeck process with the second kind}
    Let $U=\left\{U_t,
t\geq0\right\}$ be a fractional Ornstein-Uhlenbeck process with the second kind defined
as
\begin{eqnarray} \label{FOUSK} U_0=0, \mbox{ and }\ dU_{t}=-\alpha U_{t}dt+dY^{(1)}_t,\quad t\geq0,
\end{eqnarray} where  $Y^{(1)}_t=\int_0^te^{-s}dB_{a_s}$ with $a_s=He^{\frac{s}{H}}$ and $B=\left\{B_t,
t\geq0\right\}$ is a fractional Brownian motion with Hurst parameter
$H\in(\frac{1}{2},1)$, and where $\alpha>0$ is a unknown real
parameter.
\\ The equation (\ref{FOUSK}) admits an explicit solution
\begin{eqnarray*} \label{explicit FOUSK}  U_{t}&=&e^{-\alpha t}\int_0^te^{\alpha
s}dY^{(1)}_{s}=e^{-\alpha t}\int_0^te^{(\alpha-1) s}dB_{a_s}
\\&=&H^{(1-\alpha)H}e^{-\alpha t}\int_{a_0}^{a_t}r^{(\alpha-1)
H}dB_r.
\end{eqnarray*}
Hence we can also write
\begin{eqnarray} \label{decomposition FOUSK}  U_{t}=U^{\alpha}_t+R(\alpha, U^{\alpha})
\end{eqnarray}where \[U^{\alpha}_t=e^{-\alpha t}\int_{-\infty}^te^{(\alpha-1) s}dB_{a_s}
=H^{(1-\alpha)H}e^{-\alpha t}\int_{0}^{a_t}r^{(\alpha-1)
H}dB_r.\]
\begin{lemma}\label{ppt FOUSK}Let $H\in(\frac{1}{2},1)$.
Then,
\begin{eqnarray}\label{var(U_0alpha)}E\left[\left(U^{\alpha}_0\right)^2\right]=\frac{(2H-1)H^{2H}}{\alpha}\beta(1-H+\alpha H,2H-1).
\end{eqnarray}
and for large $|t|$
\begin{eqnarray}\label{cov(U_0alpha,U_talpha)}r_{U^{\alpha}}(t)=E\left[U^{\alpha}_0U^{\alpha}_t\right]=O\left(e^{-\min\{\alpha,\frac{1-H}{H}\}t}\right).
\end{eqnarray}
\end{lemma}
\begin{proof}
We prove the first point (\ref{var(U_0alpha)}). We have
\begin{eqnarray*}
E\left[\left(U^{\alpha}_0\right)^2\right]&=&H(2H-1)H^{2(1-\alpha)H}
\int_0^{a_0}dyy^{(\alpha-1) H}\int_0^{a_0}dx\ x^{(\alpha-1) H}|x-y|^{2H-2}\\&=&2H(2H-1)H^{2(1-\alpha)H}
\int_0^{a_0}dyy^{(\alpha-1) H}\int_0^{y}dx\ x^{(\alpha-1) H}(y-x)^{2H-2}\\&=&2H(2H-1)H^{2(1-\alpha)H}
\int_0^{a_0}dyy^{2\alpha H-1}\int_0^{1}dz\ z^{(\alpha-1) H}(1-z)^{2H-2}
\\&=&\frac{(2H-1)H^{2H}}{\alpha}\beta(1-H+\alpha H,2H-1).
\end{eqnarray*}
Thus (\ref{var(U_0alpha)}) is obtained.\\
For the point (\ref{cov(U_0alpha,U_talpha)}) see \cite{KS}.
\end{proof}\\
Combining (\ref{decomposition FOUSK}), Lemma \ref{asymptotic R_{Q,q}} and the fact that  $U^{\alpha}$ is ergodic we conclude that almost surely
\begin{eqnarray*}
\lim_{n\longrightarrow\infty}Q_{q,n}(U)\nonumber
&=&\lim_{n\longrightarrow\infty}Q_{q,n}(U^{\alpha}) \nonumber\\
&=&\gamma_{U^{\alpha}}(q)\nonumber\\
&=&\frac{q!}{(\frac{q}{2})! 2^{q/2}}\left(E\left[\left(U^{\alpha}_0\right)^2\right]-1\right)^{q/2}\nonumber\\
&=&\frac{q!}{(\frac{q}{2})! 2^{q/2}}\left(\frac{(2H-1)H^{2H}}{\alpha}\beta(1-H+\alpha H,2H-1)-1\right)^{q/2}\nonumber
\\&:=&\nu_q(\alpha).\label{limit a.s. 0f Q_{q,n}(U)}
\end{eqnarray*}
 Hence we obtain the following
estimator for $\alpha$
\begin{eqnarray*}
\widehat{\alpha}_{q,n}=\nu_q^{-1}\left[Q_{q,n}(U)\right].
\end{eqnarray*}
By construction, we have the following strong consistence of $\widehat{\alpha}_{q,n}$.
\begin{theorem}Let $H\in\left(\frac12,1\right)$. Then, as $n\longrightarrow\infty$
\begin{eqnarray} \widehat{\alpha}_{q,n} \longrightarrow \alpha
\label{convergence a.s. of widehat{theta}_{q,n}}\end{eqnarray} almost surely.
\end{theorem}
Combining  (\ref{decomposition FOUSK}) and Theorem \ref{CLT and NCLT of Q_{q,n}(Z-exp)}, Lemma \ref{ppt FOUSK} and Lemma \ref{asymptotic R_{Q,q}} we conclude the following result
\begin{theorem}\label{asymptotic distribution of Q_{q,n}(FOUSK)}
 Denote $N\sim\mathcal{N}(0,1)$. If $H\in(\frac12,1)$, then there exist $C$ depending on $q$, $H$ and $\alpha$ such that
\begin{eqnarray}
 d_W\left(\sqrt{\frac{n}{E\left[V_{q,n}^2(U^{\alpha})\right]}}\left(\nu_q\left(\widehat{\alpha}_{q,n}\right)-\nu_q(\alpha)\right),N\right)&\leq&  C n^{-\frac14}.
\end{eqnarray}
In particular,
\begin{eqnarray}\sqrt{n}\left(\nu_q\left(\widehat{\alpha}_{q,n}\right)-\nu_q(\alpha)\right)\overset{law}{\longrightarrow}
\mathcal{N}\left(0,\sigma^2(U^{\alpha})\right)
\end{eqnarray}where $\sigma^2(U^{\alpha})=\sum_{k=1}^{q/2}b^2_{q,2k}(U^{\alpha}) (2k)!\left(1+2\sum_{j=1}^{\infty}\frac{|r_{U^{\alpha}}(j)|^{2k}}{r_{U^{\alpha}}^2(0)}\right)=\lim_{n\rightarrow\infty}E\left[V_{q,n}^2(U^{\alpha})\right]$.
\end{theorem}

\begin{theorem}\label{asymptotic distrivution of widehat{theta}_{q,n}}
Let $H\in(\frac12,1)$. Then
\begin{eqnarray}\sqrt{\frac{n}{E\left[V_{q,n}^2(U^{\alpha})\right]}}\left(\widehat{\alpha}_{q,n}-\alpha\right)\overset{law}{\longrightarrow}\mathcal{ N}\left(0, \left(\nu'_q(\alpha)\right)^{-2}\right).
\end{eqnarray}
\end{theorem}
\begin{proof} We can write
\begin{eqnarray*}\sqrt{n}\left(\nu_q(\widehat{\alpha}_{q,n})-\nu_q(\alpha)\right)
&=&\nu'_q(\xi_{q,n})\sqrt{n}\left(\widehat{\alpha}_{q,n}-\alpha\right)
\end{eqnarray*}
where $\xi_{q,n}$ is a random variable between $\alpha$ and $\widehat{\alpha}_{q,n}$.\\
Combining this with Theorem \ref{asymptotic distribution of Q_{q,n}(FOUSK)} we obtain the desired conclusion.
\end{proof}

\vspace{.5cm}
{\bf Quadradic case:}
Here we suppose that $q=2$. We will study the asymptotic distribution of $\widehat{\alpha}_{2,n}$.
\begin{theorem}Let $H\in\left(\frac12,1\right)$. Then there exists $C$ depending on $H$ and $\alpha$ such that
\begin{equation}   \label{d_W(F_{2,n}(U),N) for alpha}
 d_W\left(\sqrt{\frac{n}{E\left[V_{2,n}^2(U^{\alpha})\right]}}\left(\nu_2\left(\widehat{\alpha}_{2,n}\right)-\nu_2(\alpha)\right),N\right)\leq \frac{C }{\sqrt{n}}.
\end{equation}
\end{theorem}
\begin{proof}From Lemma \ref{ppt FOUSK} we have for any $H\in(\frac12,1)$,  $\sum_{k\in\mathbb{Z}}|r_{Z^{\alpha}}(k)|^{3/2}<\infty$ and $\sum_{k\in\mathbb{Z}}|r_{Z^{\alpha}}(k)|^{2}<\infty$. Hence by Theorem \ref{CLT and NCLT of Q_{2,n}(Z-exp)} the conclusion is obtained.
\end{proof}

\end{document}